\documentclass{article}

\usepackage{fancyhdr}
\usepackage{amsmath,amsthm,amssymb}
\usepackage{graphicx}
\usepackage{hyperref}

\newtheorem{theorem}{Theorem}
\newtheorem{lemma}[theorem]{Lemma}

\theoremstyle{definition}

\theoremstyle{remark}

\numberwithin{equation}{section}

\begin{document}

\title{Perfect Parallelepipeds Exist}


\author{Jorge F. Sawyer\cr
Box 8681 Farinon Center, Lafayette College, Easton, PA 18042\cr
\href{maileto: sawyerj@lafayette.edu}{e-mail: sawyerj@lafayette.edu}
\and
Clifford A. Reiter\cr
Department of Mathematics, Lafayette College, Easton, PA 18042\cr
\href{mailto: reiterc@lafayette.edu}{e-mail: reiterc@lafayette.edu}\cr
}


\date{}


\maketitle

\begin{abstract}
There are parallelepipeds with edge lengths, face diagonal lengths and body diagonal lengths all positive integers. In particular, there is a parallelepiped with edge lengths $271$, $106$, $103$, minor face diagonal lengths $101$, $266$, $255$, major face diagonal lengths $183$, $312$, $323$, and body diagonal lengths $374$, $300$, $278$, $272$. Focused brute force searches give dozens of primitive perfect parallelepipeds. Examples include parallellepipeds with up to two rectangular faces.
\end{abstract}

\section{Introduction}
A famous open problem in Number Theory is whether there exists a perfect cuboid. That is, is there a rectangular box in $\mathbb{R}^3$ with positive integer edge lengths, face diagonal lengths and body diagonal lengths \cite{Guy,Leech}? In \cite{Guy} Richard Guy poses the weaker question of whether there exist perfect parallellepipeds in $\mathbb{R}^3$. A perfect parallelepiped is a parallelepiped with edge lengths, face diagonal lengths and body diagonal lengths all positive integers. Previous attempts at finding perfect parallelepipeds focused on using rational coordinates 
\cite{Dargenio, Reiter,
Tirrell}. Here we show that perfect parallelepipeds exist by 
giving examples and we describe a technique using necessary 
conditions within brute force searches that check at the last stage whether 
proposed perfect parallelepipeds can be realized in $\mathbb{R}^3$. 

\section{There is a Perfect Parallelepiped}
While we will discuss our search strategy in the next section, it is straighforward to
exhibit and verify that a perfect parallelepiped exists, which is our main result. We call the shorter diagonal of a parallegram the minor diagonal and the longer diagonal the major diagonal. These will be the same for a rectangle. 

\begin{theorem} There is a perfect parallelepiped with edge lengths $271$, $106$, $103$, minor face diagonal lengths $101$, $266$, $255$, major face diagonal lengths $183$, $312$, $323$, and body diagonal lengths $374$, $300$, $278$, $272$. 
\end{theorem}
\begin{proof}
Consider the parallelepiped with edge vectors given by $\vec{u}=\langle 271,0,0\rangle$,
$\vec{v}=\langle\frac{9826}{271},\frac{60\sqrt{202398}}{271},0\rangle$,
$\vec{w}=\langle\frac{6647}{271},\frac{143754}{271}\sqrt{\frac{42}{4819}},66\sqrt{\frac{8358}{4819}}\rangle$. Direct computation verifies that 
$\|\vec{u}\|=271$, 
$\|\vec{v}\|=106$, 
$\|\vec{w}\|=103$, 
$\|\vec{u}-\vec{v}\|=255$, 
$\|\vec{u}-\vec{w}\|=266$, 
$\|\vec{v}-\vec{w}\|=101$, 
$\|\vec{u}+\vec{v}\|=323$, 
$\|\vec{u}+\vec{w}\|=312$, 
$\|\vec{v}+\vec{w}\|=183$, 
$\|\vec{u}+\vec{v}+\vec{w}\|=374$, 
$\|\vec{u}+\vec{v}-\vec{w}\|=300$, 
$\|\vec{u}-\vec{v}+\vec{w}\|=278$, and 
$\|-\vec{u}+\vec{v}+\vec{w}\|=272$. 
\end{proof}
A $Mathematica^{TM}$ script verifying those computations may be found at \cite{Sawyer}. Additional examples of perfect parallelepipeds may also be found there. These include
parallelepipeds with two rectangular faces. The parallelepiped with edge vectors $\vec{u}=\langle 1120,0,0\rangle$,
$\vec{v}=\langle 0,1035,0\rangle$,
$\vec{w}=\langle 0,\frac{46548}{115},\frac{12}{115}\sqrt{49755859}\rangle$ has that form. In particular, it has edge lengths 
$\|\vec{u}\|=1120$, 
$\|\vec{v}\|=1035$ and 
$\|\vec{w}\|=840$. The rectangular face diagonal lengths are
$\|\vec{u}+\vec{v}\|=\|\vec{u}-\vec{v}\|=1525$
and $\|\vec{u}+\vec{w}\|=\|\vec{u}-\vec{w}\|=1400$
and the other face has diagonal lengths 
$\|\vec{v}-\vec{w}\|=969$ and $\|\vec{v}+\vec{w}\|=1617$. 
The body diagonals lengths are 
$\|\vec{u}+\vec{v}+\vec{w}\|=\|\vec{u}-\vec{v}-\vec{w}\|=1967$ and 
$\|\vec{u}+\vec{v}-\vec{w}\|=\|\vec{u}-\vec{v}+\vec{w}\|=1481$. 
 
\section{The Search}
First we observe that the major diagonal of a parallelogram can be expressed in terms of the edges and the minor diagonal. That can be used to facilitate the search for perfect parallelograms; that is, parallelograms with edge and diagonal lengths that are all positive integers.

\begin{lemma}
Let $x_1$, $x_2$ , and $d_{12}$ be positive integers with $1 \le x_2 \le x_1$ and $x_1-x_2 < d_{12} \le \sqrt{x_1^2+x_2^2}$. Then the parallelogram with edge length $x_1$ and $x_2$ and minor diagonal length $d_{12}$ is perfect if and only if $2x_1^2+2x_2^2-d_{12}^2$ is a square.
\end{lemma}
\begin{proof}
Let $\vec u$ and $\vec v$ be edge vectors for the parallelogram so that $\|\vec u\|=x_1$, $\|\vec v\|=x_2$, $\|\vec u-\vec v\|=d_{12}$, then the result follows from observing that
$\|\vec u+\vec v\|^2=2\|\vec u\|^2+2\|\vec v\|^2-\|\vec u-\vec v\|^2$
\end{proof}

The search technique that we used determined by brute force all such $x_1$, $x_2$, $d_{12}$ where $x_1$ was below some bound. Then all non-oblique assemblies of three such perfect parallelograms 
with matching pairs of edges were 
considered as possible perfect parallelepipeds. The search was implemented in J \cite{Jsoftware}. 
\par
Whether the body diagonals were of integer 
length was determined by the following lemma. We use $d_{ij}$ to denote the minor diagonal length of the parallelogram with edges $i$ and $j$. The body diagonal with edge $i$ 
having negative contribution is denoted $m_i$, $1\le i\le 3$ and $m_4$ denotes the length of the body diagonal when all edges contribute positively.

\begin{lemma}
Suppose there is a parallelepiped with edge lengths $x_1$, $x_2$, and $x_3$ and minor face diagonal lengths $d_{12}$, $d_{13}$ and $d_{23}$. Then the square of the body diagonal lengths are $m_1^2=-x_1^2+x_2^2+x_3^2+d_{12}^2+d_{13}^2-d_{23}^2$,
$m_2^2=x_1^2-x_2^2+x_3^2+d_{12}^2-d_{13}^2+d_{23}^2$, 
$m_3^2=x_1^2+x_2^2-x_3^2-d_{12}^2+d_{13}^2+d_{23}^2$, and 
$m_4^2=3x_1^2+3x_2^2+3x_3^2-d_{12}^2-d_{13}^2-d_{23}^2$. 
\end{lemma}
\begin{proof}
Let $\vec{u}$, $\vec{v}$, $\vec{w}$ be edge vectors for the parallelepiped such that
$\|\vec{u}\|=x_1$, $\|\vec{v}\|=x_2$, $\|\vec{w}\|=x_3$,
 $\|\vec{u}-\vec{v}\|=d_{12}$,
 $\|\vec{u}-\vec{w}\|=d_{13}$,
 $\|\vec{v}-\vec{w}\|=d_{23}$. 
 Note that $2\vec{u} \cdot \vec{v}=x_1^2+x_2^2-d_{12}^2$ and likewise for the other dot products.
 We see that
 \begin{align}
 \|-\vec{u}+\vec{v}+\vec{w}\|^2 &=\|\vec{u}\|^2+ \|\vec{w}\|^2+\|\vec{w}\|^2-2\vec{u}\cdot\vec{v}-2\vec{u}\cdot\vec{w}+2\vec{v}\cdot\vec{w}\nonumber\\
&=x_1^2+ x_2^2+x_3^2-(x_1^2+x_2^2-d_{12}^2)-(x_1^2+x_3^2-d_{13}^2)\nonumber\\
&~~~~~~~~~~~~~~~~~~~~~~~~~~~~+(x_2^2+x_3^2-d_{23}^2)\nonumber\\
&=-x_1^2+x_2^2+x_3^2+d_{12}^2+d_{13}^2-d_{23}^2\nonumber
\end{align}
 as desired. The other cases are similar.
\end{proof}
Our search quickly located triples of perfect parallelograms with matching edge lengths $x_1$, $x_2$, and $x_3$ and minor diagonal lengths $d_{12}$, $d_{13}$ and $d_{23}$ that also had all four 
proposed body diagonals $m_1$, $m_2$, $m_3$ and $m_4$ of positive integer length. For example,
the smallest such is given by $x_1=115$, $x_2=106$, $x_3=83$, $d_{12}=31$, $d_{13}=58$ and $d_{23}=75$. However, these perfect parallelograms cannot be realized as a parallelepiped in $\mathbb{R}^3$.
\par
The following lemma gives the final criterion necessary for the assembly to be realizable. We let $\theta_{ij}$ denote the angle between edges $x_i$ and $x_j$ in the triangle with sides $x_i$, $x_j$ and $d_{ij}$ and let $c_{ij}$ denote the cosine of that angle. Note that $c_{ij}=cos(\theta_{ij})=\frac{x_i^2+x_j^2-d_{ij}^2}{2x_i x_j}$ and by our choice of minor diagonal $0\le c_{ij}<1$.

\begin{lemma}
An edge-matched assembly of three perfect parallelograms with edge lengths $x_1$, $x_2$, and $x_3$ and minor diagonal lengths $d_{12}$, $d_{13}$ and $d_{23}$ can be assembled in $\mathbb{R}^3$ into a parallelepiped if $c_{12}^2+c_{13}^2+c_{23}^2<1+2c_{12} c_{13} c_{23}$.  
\end{lemma}
\begin{proof}
Let $\vec{u}=x_1\langle 1,0,0 \rangle$, 
$\vec{v}=x_2\langle c_{12},\sqrt{1-c_{12}^2},0\rangle$, $\rho=\frac{c_{23}-c_{12}c_{13}}{\sqrt{1-c_{12}^2}\sqrt{1-c_{13}^2}}$, and 
$\vec{w}=x_3\langle c_{13},\rho \sqrt{1-c_{13}^2},\sqrt{1-\rho^2}\sqrt{1-c_{13}^2}\rangle$. Direct computation shows that
that the parallelepiped generated by $\vec{u}$,  $\vec{v}$,  $\vec{w}$ realizes the parallelepiped with desired edges and minor diagonals 
provided that $-1<\rho<1$. Note that $p=\pm 1$ would yield a degenerate parallelepiped. A $Mathematica^{TM}$ script verifying those computations may be found at \cite{Sawyer}. The condition that $-1<\rho<1$ is equivalent to $\rho^2<1$ which is equivalent to 
$(c_{23}-c_{12}c_{13})^2<(1-c_{12}^2)(1-c_{13}^2)$ and that simplifies to the required inequality.  
\end{proof}
\par
The above lemma describes realizability using non-oblique assemblies at one vertex. Note that moving along any edge of such a perfect parallelepiped leads to a vertex
with two angles becoming non-acute. Thus, configurations with an odd number of oblique
angles at each vertex would be distinct from those above and these too exist \cite{Sawyer}.
At least the first two of those were first found by Randall Rathbun\cite{Rathbun}.   
\par
To give some sense of the number of edge-matched non-oblique configurations checked we offer sample statisitics. When checking edges up to $3949$ there were about $2\times10^{10}$ 
non-oblique edge-matched configurations tested. Of those, about $9\times 10^7$ satisfied one of the necessary body diagonal conditions from Lemma 3. About $1.7\times 10^6$ satisfied two; $33403$ satisfied three; $414$ satisfied all four. Of those, $27$ gave realizable perfect parallelipipeds.
\par
We have established that perfect parallelepipeds exist, and some with two rectangular faces exist. The question of whether perfect cuboids exist remains open. Intermediate questions are also open. Is there a perfect parallelepiped with integer volume? Is there a perfect parallelepiped with rational coordinates? 
\par
Acknowledgement: The support of a Lafayette EXCEL grant is appreciated.
\bibliographystyle{amsplain}

\end{document}